\documentclass{amsart}
\usepackage{amsfonts}

\usepackage{amscd,amssymb,amsmath,graphicx,verbatim}
\usepackage[dvips]{hyperref}
\usepackage[TS1,OT1,T1]{fontenc}

\newtheorem{theorem}{Theorem}[section]
\newtheorem{lemma}[theorem]{Lemma}

\newtheorem{proposition}[theorem]{Proposition}

\theoremstyle{definition}
\newtheorem{definition}[theorem]{Definition}
\newtheorem{example}[theorem]{Example}

\theoremstyle{remark}
\newtheorem{remark}[theorem]{Remark}
\numberwithin{equation}{section}

\begin{document}

\title{Some dynamical properties for linear operators}
\author{Bingzhe Hou }
\address{Bingzhe Hou, Department of Mathematics , Jilin university, 130012, Changchun, P.R.China} \email{houbz@jlu.edu.cn}

\author{Geng Tian}
\address{Geng Tian, Department of Mathematics , Jilin university, 130012, Changchun, P.R.China} \email{tian\_geng@yahoo.com.cn}

\author{Luoyi Shi}
\address{Luoyi Shi, Department of Mathematics , Jilin university, 130012, Changchun, P.R.China} \email{shiluoyi811224@sina.com}

\date{Mar. 26, 2009}
\thanks{The first author
is supported by the Youth Foundation of Department of Mathematics,
Jilin university.}
\begin{abstract}
In our another recent article, we introduce a new dynamical property
for linear operators called norm-unimodality which implies
distributional chaos.  In the present paper, we'll give a further
discussion of norm-unimodality. It is showed that norm-unimodality
is similar invariant and the spectra of norm-unimodal operator is
referred to. As an application, in each nest algebra there exist
distributional chaotic operators. Moreover, normal operators and
compact operators with regard to norm-unimodality and Li-Yorke chaos
are also be considered. Specially,  a small compact perturbation of
the unit operator could be distributionally chaotic.
\end{abstract}
\maketitle

\section{Introduction and Preliminaries}

A discrete dynamical system is simply a continuous mapping $f:
X\rightarrow X$ where $X$ is a complete separable metric space. For
$x\in X$, the orbit of $x$ under $f$ is
$Orb(f,x)=\{x,f(x),f^{2}(x),\ldots\}$ where $f^{n}= f\circ f\circ
\cdots \circ f $ is the $n^{th}$ iterate of $f$ obtained by
composing $f$ with $n$ times.

In 1975, Li and Yorke \cite{L-Y} observed complicated dynamical
behavior for the class of interval maps with period 3. This
phenomena is currently known under the name of Li-Yorke chaos.
Therefrom, several kinds of chaos were well studied. In the present
article, we focus on distributional chaos.

\begin{definition}
$\{x,y\}\subset X$ is said to be a Li-Yorke chaotic pair, if
\begin{equation*}
\limsup\limits_{n\rightarrow\infty}d(f^{n}(x),f^{n}(y))>0, \quad
\liminf\limits_{n\rightarrow\infty}d(f^{n}(x),f^{n}(y))=0.
\end{equation*}
Furthermore, $f$ is called Li-Yorke chaotic, if there exists an
uncountable subset $\Gamma\subseteq X$ such that each pair of two
distinct points in $\Gamma$ is a Li-Yorke chaotic pair.
\end{definition}

From Schweizer and Sm\'{\i}tal's paper \cite{S-S}, distributional
chaos is defined in the following way.

For any pair $\{x,y\}\subset X$ and any $n\in \mathbb{N}$, define
distributional function $F^{n}_{xy}:\mathbb{R}\rightarrow [0,1]$:
\begin{equation*}
F^{n}_{xy}(\tau)=\frac{1}{n}\#\{0\leq i\leq n-1:
d(f^{i}(x),f^{i}(y))<\tau\}.
\end{equation*}
Furthermore, define
\begin{align*}
F_{xy}(\tau)=\liminf\limits_{n\rightarrow\infty}F^{n}_{xy}(\tau), \\
F_{xy}^{*}(\tau)=\limsup\limits_{n\rightarrow\infty}F^{n}_{xy}(\tau)
\end{align*}
Both $F_{xy}$ and $F_{xy}^{*}$ are nondecreasing functions and may
be viewed as cumulative probability distributional functions
satisfying $F_{xy}(\tau)=F_{xy}^{*}(\tau)=0$ for $\tau<0$.

\begin{definition}
$\{x,y\}\subset X$ is said to be a distributionally chaotic pair, if
$$
F_{xy}^{*}(\tau)\equiv 1, \ \ \forall \ \ \tau>0 \quad and \quad
F_{xy}(\epsilon)=0, \ \ \exists \ \ \epsilon>0.
$$
Furthermore, $f$ is called distributionally chaotic, if there exists
an uncountable subset $\Lambda\subseteq X$ such that each pair of
two distinct points in $\Lambda$ is a distributionally chaotic pair.
Moreover, $\Lambda$ is called a distributionally
$\epsilon$-scrambled set.
\end{definition}

Distributional chaos always implies Li-Yorke chaos, as it requires
more complicated statistical dependence between orbits than the
existence of points which are proximal but not asymptotic. The
converse implication is not true in general. However in practice,
even in the simple case of Li-Yorke chaos, it might be quite
difficult to prove chaotic behavior from the very definition. Such
attempts have been made in the context of linear operators (see
\cite{Duan, Fu}). Further results of \cite{Duan} were extended in
\cite{Opr} to distributional chaos for the annihilation operator of
a quantum harmonic oscillator. More about distributional chaos, one
can see \cite{Smital1, Smital2, Liao1, Liao2, Wang}.

We are interested in the dynamical systems induced by continuous
linear operators on Banach spaces. From Rolewicz's article
\cite{Rolewicz}, hypercyclicity is widely studied. In fact, it
coincides a dynamical property "transitivity". Now there has been
got so many improvements at this aspect (Grosse-Erdmann's and
Shapiro's articles \cite{Grosse,Shapiro} are good surveys.).
Specially, distributional chaos for shift operators were discussed
by F. Mart\'{\i}nez-Gim\'{e}nez, et.al. in \cite{Gim}. In a recent
article \cite{Hbz} of the first author, one introduce a new
dynamical property for linear operators called norm-unimodality
which implies distributional chaos.

\begin{definition}
Let $X$ be a Banach space and let $T\in \mathcal {L}(X)$. $T$ is
called norm-unimodal, if we have a constant $r
>1$ such that for any $m\in\mathbb{N}$, there exists $x_m\in
X$ satisfying
$$
\lim\limits_{k\rightarrow\infty}\|T^kx_m\|=0, \ \ and \ \ \| T^ix_m
\|\geq r^i\|x_m\|, \ \ i=1,2,\ldots,m.
$$
Furthermore, such $r$ is said to be a norm-unimodal constant for the
norm-unimodal operator $T$.
\end{definition}

\begin{theorem}[Distributionally Chaotic Criterion \cite{Hbz}]\label{D-C-C}
Let $X$ be a Banach space and let $T\in \mathcal {L}(X)$. If $T$ is
norm-unimodal, then $T$ is distributionally chaotic.
\end{theorem}

More generally,

\begin{theorem}[Weakly Distributionally Chaotic Criterion \cite{Hbz}]\label{W-D-C-C}
Let $X$ be a Banach space and let $T\in \mathcal {L}(X)$. Suppose
$C_m$ be a sequence of positive numbers  increasing to $+\infty$. If
there exist $\{x_m\}_{m=1}^{\infty}$ in $X$ satisfying

$(WNU1) \ \ \  \ \ \lim\limits_{k\rightarrow\infty}\|T^kx_m\|=0$.

$(WNU2)$  \ \ There is a sequence of positive integers $N_m$
increasing to $+\infty$, such that
$\lim\limits_{m\rightarrow\infty}\frac{\# \{0\leq i\leq N_m-1; \|
T^ix_m \|\geq C_m\|x_m\| \} }{N_m}=1$.

Then $T$ is distributionally chaotic.

\end{theorem}

In the present paper,  we'll show that norm-unimodality is similar
invariant firstly, and give a description for the spectra of
norm-unimodal operators. As an application of Theorem \ref{D-C-C},
in each nest algebra there exist distributional chaotic operators.
At the end, normal operators and compact operators are referred to.
It is proven that neither normal operator nor compact operator is
Li-Yorke chaotic, and any compact perturbation of $\lambda I$ can't
be norm-unimodal. Surprisingly, a small compact perturbation of the
unit operator could be distributionally chaotic.

\section{Norm-unimodal operators}

\begin{proposition}
Let $X$ be a Banach space. let $T , C \in \mathcal {L}(X)$ and $C$
be an invertible operator. If $T$ is  norm-unimodal, then $C^{-1}TC$
is also norm-unimodal.
\end{proposition}

\begin{proof}
Firstly, claim that the following statements are equivalent.

$(a)$.  $T$ is norm-unimodal.

$(b)$.  There are constants $r
>1$ and $N \in \mathbb{N}$  such that for any $m\geq N$, there exist
$x_m\in X$ satisfying
$$
\lim\limits_{k\rightarrow\infty}\|T^kx_m\|=0,\ \ and \ \ || T^ix_m
||\geq r^i\|x_m\|, \ \ i=N,N+1,\ldots,m.
$$
For convenience, denote the property $(b)$ as $P(N)$.

$(a) \Rightarrow (b)$ is obvious.

$(b) \Rightarrow (a)$:  If $P(N)$ holds, then for any $m \geq N-1$,
there exist $x_{m} \in \mathbb{N}$ satisfying
$$
\| T^ix_m \|\geq r^i\|x_m\|, \ \ i=N,N+1,\ldots,m+N-1, \ \  and
 \ \ \lim\limits_{k\rightarrow\infty}\|T^kx_m\|=0.
$$
For $N-1$, it is either $||T^{N-1}x_{m}|| \geq r^{N-1}||x_{m}|| $ or
$||T^{N-1}x_{m}|| < r^{N-1}||x_{m}|| $.

If $||T^{N-1}x_{m}|| < r^{N-1}||x_{m}|| $,  set
$y_{m}=T^{N-1}x_{m}$. Then we have
\begin{eqnarray*}
&&||T^{N-1}y_{m}|| \geq r^{N-1}||r^{N-1}x_{m}|| >
r^{N-1}||T^{N-1}x_{m}||=r^{N-1}||y_{m}||, \\
&&||T^{i}y_{m}||=||T^{i+N-1}x_{m}||\geq r^{i}||y_{m}|| , N \leq i
\leq m, \ \  and \\
&&\lim\limits_{k\rightarrow\infty}\|T^k y_m\|=0.
\end{eqnarray*}
So in any case, $P(N-1)$ holds, i.e., there exist $z_{m} \in X$ such
that
$$
||T^{i}z_{m}|| \geq r^{i}||z_{m}|| , N-1 \leq i \leq m ,  \ \ and \
\ \lim\limits_{k\rightarrow\infty}\|T^kz_m\|=0.
$$
Continue in this matter finitely, one can see $(a)$ holds.

Now it's sufficient to show that $C^{-1}TC$ satisfies $(b)$.

Suppose $T$ is norm-unimodal. Then there is a constant $r
>1$ such that for any $m\in\mathbb{N}$, there exists $x_m\in
X$ satisfying
$$
\lim\limits_{k\rightarrow\infty}\|T^kx_m\|=0, \ \ and \ \ \| T^ix_m
\|\geq r^i\|x_m\|, \ \ i=1,2,\ldots,m.
$$
Choose $r_{1}$ such that $1< r_{1} < r$. Then there exist $N_{1}$
such that for any $k\geq N_{1}$, $(r/r_{1})^{k}\geq
||C||\cdot||C^{-1}||$. Given any $m\geq N_{1}$ , we have for
$N_{1}\leq i \leq m$,
$$||(C^{-1}TC)^{i}C^{-1}x_{m}||\geq
\frac{1}{||C||}||T^{i}x_{m}||\geq \frac{r^{i}}{||C||}||x_{m}||\geq
r_{1}^{i}||C^{-1}||\cdot||x_{m}||\geq r_{1}^{i}||C^{-1}x_{m}||,$$
and
$$\lim\limits_{k\rightarrow\infty}||(C^{-1}TC)^{k}C^{-1}x_{m}||=0.$$
Therefore, norm-unimodality is similar invariant.
\end{proof}

Next, let's consider the the spectra of norm-unimodal operators.
Denote $\mathbb{D}$ be the unit open disk on complex plane.
Moreover, $\mathbb{D}^{-}$ denoted as its closure and
$\partial\mathbb{D}$ denoted as its boundary.

\begin{proposition}\label{spec}
Let $T$ be a norm-unimodal operator on complex Hilbert space. Then
there exist $\{\lambda_{n}\}_{n=1}^{\infty} \subset \sigma(T) \cap
{\mathbb{D}^{-}}^{c}$ such that
$\lim\limits_{n\rightarrow\infty}\lambda_{n}=\lambda \in
\partial\mathbb{D}$.
\end{proposition}

\begin{proof}
Suppose it's not true.  Then according to Riesz Decomposition
Theorem, we can obtain
$$T=\begin{matrix}\begin{bmatrix}
T_{1}\\
&T_{2}\\
\end{bmatrix}&
\begin{matrix}
H_1\\
 H_2\end{matrix}\end{matrix},$$
where $\sigma(T_{1})=\sigma(T)\cap \mathbb{D}^{-}$,
$\sigma(T_{2})=\sigma(T)-\sigma(T_{1})$ and $\sigma(T_{2})\subseteq
\{z; |z|\geq \delta > 1 \}$

Furthermore,
$$T=\begin{matrix}\begin{bmatrix}
T_{1}&*\\
&\widetilde{T_{2}}\\
\end{bmatrix}&
\begin{matrix}
H_1\\
  H_1^{\perp}\end{matrix}\end{matrix},$$
where $\sigma(\widetilde{T_2})=\sigma(T_{2}).$

Since T is norm-unimodal, then there is a constant $r>1$ such that
for any $m \in \mathbb{N}$, there exist $x_{m} \in H$ satisfying
$$
\lim\limits_{k\rightarrow\infty}\|T^kx_m\|=0, \ \ and \ \ \| T^ix_m
\|\geq r^i\|x_m\|, \ \ i=1,2,\ldots,m.
$$
In addition,
$$x_{m}=x_{m}^{1}+x_{m}^{2}, \ \ \ where \ \  x_{m}^{1}\in H_{1}, \ \ x_{m}^{2}\in H_{1}^{\perp} ,$$ then
$\lim\limits_{k\rightarrow\infty}\|\widetilde{T_2}^kx_m^2\|=0$.

According to Spectral Mapping Theorem and Spectral Radius Formula,
$$r_{1}(\widetilde{T_2})^{-1}=r(\widetilde{T_2}^{-1})=\lim\limits_{k\rightarrow\infty}\|\widetilde{T_2}^{-k}\|^{\frac{1}{k}}.$$
Note $r_{1}(\widetilde{T_2})\geq \delta > 1$, one can choose
$\epsilon
> 0$ such that $r_{1}(\widetilde{T_2})^{-1}+\epsilon < 1$. Then
there exist $M \in \mathbb{N}$ such that for $k\geq M$,
$$
\frac{1}{\|\widetilde{T_2}^{-k}\|} \geq
(\frac{1}{r_{1}(\widetilde{T_2})^{-1}+\epsilon})^{k}.
$$
Since as $k\rightarrow\infty$,
$$
0\leftarrow ||\widetilde{T_2}^kx_m^2||\geq
\frac{1}{||\widetilde{T_2}^{-k}||}||x_{m}^{2}|| \geq
(\frac{1}{r_{1}(\widetilde{T_2})^{-1}+\epsilon})^{k}\|x_{m}^{2}\|\geq
\|x_{m}^{2}\|,
$$
then $x_{m}^{2}=0$ and hence $T^{k}x_{m}=T_{1}^{k}x_{m}^{1}$.
Consequently, for $1\leq i \leq m$,
$$
||T_{1}^{i}x_{m}^{1}||=||T^{i}x_{m}||\geq
r^{i}||x_{m}||=r^{i}||x_{m}^{1}||.
$$
So $r(T_{1}) \geq r >1$. It is a contradiction.
\end{proof}
Speaking intuitively, in the spectra of each norm-unimodal operator
there should be a sequence of points outside the unit circle
converges to a point on the unit circle. And there exist
norm-unimodal operators whose spectra disjoint with the unit open
disk.
\begin{example}
Let $T$ be a bilateral weighted shift operator with weights
$\{\omega_n\}_{n\in \mathbb{Z}}$ as
$$
\omega_n=\left\{\begin{array}{cc}
2, \ \ \ &\mbox{if \ $n\geq0$} \\
\dfrac{|n|-1}{|n|}, \ \ \ &\mbox{if \ $n<0$}
\end{array}\right.
$$
Obviously, $T$ is norm-unimodal but $\sigma(T)\cap \mathbb{D}=\phi$.
\end{example}

From theorem \ref{D-C-C}, we see norm-unimodality implies
distributional chaos. However, there exist distributionally chaotic
operators but not norm-unimodal and its spectra coincides the unit
circle.

\begin{example}
Let $T$ be a bilateral weighted shift operator with weights
$\{\omega_n\}_{n\in \mathbb{Z}}$ as
$$
\omega_n=\left\{\begin{array}{cc}
\dfrac{n+2}{n+1}, \ \ \ &\mbox{if \ $n\geq0$} \\
\dfrac{|n|-1}{|n|}, \ \ \ &\mbox{if \ $n<0$}
\end{array}\right.
$$
Obviously, $T$ is distributional chaotic and $\sigma(T)=
\partial\mathbb{D}$, although it is not norm-unimodal.
\end{example}

At the end of this section, let's consider distributional chaos in
nest algebra via the technique of norm-unimodality.

\begin{definition}
A nest $\mathfrak{N}$ is a chain of closed subspaces of a Hilbert
space $H$ containing $\{0\}$ and $H$ which is closed under
intersection and closed span. The nest algebra $\mathcal
{T}(\mathfrak{N})$ of a given nest $\mathfrak{N}$ is the set of all
operators $T$ such that $TN \subseteq N$ for every element $N$ in
$\mathfrak{N}$.
\end{definition}

\begin{definition}
Given a nest $\mathfrak{N}$. For $N$ belonging to $\mathfrak{N}$,
define  $$N_{-}=\bigvee \{N^{'} \in \mathfrak{N} ; \  N^{'}< N \}$$
If $N_{-} \neq N$, then we call $N_{-}$ the immediate predecessor to
$N$, otherwise $N$ has no immediate predecessor.
\end{definition}

\begin{lemma}\label{lem1}
$\mathcal {T}(\mathfrak{N})$ is a weak operator closed subalgebra of
$B(H)$.
\end{lemma}

The proof can be referred to \cite{Davidson}.

\begin{proposition}
Given a nest $\mathfrak{N}$ ,  then there exist an operator T in
$\mathcal {T}(\mathfrak{N})$ such that T is distributionally
chaotic.
\end{proposition}
\begin{proof}
If $H$ has the immediate predecessor $H_{1}$, then $H_{1}$ will be
considered. Continue in this manner, we may obtain  either a chain
$H > H_{1} > H_{2} > \cdots $, or a $H_{M}$ which has no immediate
predecessor.

The first case. Let $F_{i}=H_{i-1}\ominus H_{i}$ and $H_{0}=H$.
Choose $f_{i} \in F_{i}$ such that $\|f_{i}\|=1$. Define
$$
A=
\underset{n=1}{\overset{+\infty}{\sum}}\underset{i=0}{\overset{n-1}{\sum}}f_{a_{n}+i+1}\otimes2f_{a_{n}+i},
$$
where $a_{1}=1$ and $a_{n}-a_{n-1}=n$. Since each $f_{a_{n}+i+1}
\otimes 2f_{a_{n}+i}$ belongs to $\mathcal{T}(\mathfrak{N})$, then
$A$ belongs to $\mathcal{T}(\mathfrak{N})$ by lemma \ref{lem1}.

The second case. Since $H_{M}$ has no immediate predecessor,there
exist a chain $N_{1}<N_{2}<\cdots<N_{i}<\cdots<H_{M}$ in
$\mathfrak{N}$ such that $SOT-\lim_{i \rightarrow
\infty}P(N_{i})=P(H_{M})$. Let $E_{i}=N_{i+1}\ominus N_{i}$,  $i\geq
1$. Choose\ $e_{i}\in E_{i}$ such that ${\parallel e_{i}
\parallel}=1$. Define
$$
B=
\underset{n=1}{\overset{+\infty}{\sum}}\underset{i=0}{\overset{n-1}{\sum}}e_{a_{n}+i}\otimes2e_{a_{n}+i+1},
$$
where $a_{1}=1$ and $a_{n}-a_{n-1}=n$. Since each
$e_{a_{n}+i}\otimes2e_{a_{n}+i+1}$ belongs to
$\mathcal{T}(\mathfrak{N})$, then  $A$ belongs to
$\mathcal{T}(\mathfrak{N})$ by lemma \ref{lem1}.

One can easily see that both $A$ and $B$ have the form of matrix
under suitable bases as follows :
$$\begin{matrix}\begin{bmatrix}
W_{1}\\
&W_{2}\\
&&\ddots\\
&&&W_{k-1}\\
&&&&\ddots
\end{bmatrix}&
\begin{matrix}
\end{matrix}\end{matrix},$$
where $${W_{k-1}=\begin{matrix}\begin{bmatrix}
0&2\\
&\ddots&\ddots\\
&&\ddots&2\\
&&&0\\
\end{bmatrix}&
\begin{matrix}
\end{matrix}\end{matrix}}_{(k\times k)}.$$

$T$ could seem as this matrix, so one can easily prove that $T$ is
norm-unimodal and hence is distributional chaos.
\end{proof}

\section{Normal operators and compact operators}

In this section, we'll consider norm-unimodality and Li-Yorke chaos
for normal operators and compact operators.

\begin{proposition}
Let $N$ be a normal operator on separable complex Hilbert space.
Then $N$ is impossible to be Li-Yorke chaotic. Consequently, $N$ is
neither distributionally chaotic nor norm-unimodal.
\end{proposition}

\begin{proof}
Since $N$ is normal, there exist finite positive regular Borel
measure $\mu$ and Borel function $\eta \in L^{\infty}(\sigma(N) ,
\mu)$ such that $N$ and $M_{\eta}$ are unitarily equivalent.
$M_{\eta}$ is multiplication by $\eta$ on $L^{2}(\sigma(N) , \mu)$.
To see $M_{\eta}$ being not Li-Yorke chaotic, it's sufficient to
prove $\lim\limits_{m\rightarrow\infty}\|M_{\eta}^m(f)\|=0$ if $0\in
\omega(f)$.

Let
\begin{eqnarray*}
&&\Delta_{1}=\{z\in \sigma(N) ; |\eta(z)|\geq 1 \} ,\\
&&\Delta_{2}=\{z\in \sigma(N) ; |\eta(z)|< 1 \},\\
&&\Delta_{3}=\{z\in \sigma(N) ; f(z)=0 \ \ a.e. \ [\mu] \},\\
&&\Delta_{4}=\{z\in \sigma(N) ; f(z)\neq 0 \ \ a.e. \ [\mu] \}.
\end{eqnarray*}
Since $0\in \omega(f)$, there exist $\{m_{k}\}_{k=1}^{\infty} $ such
that $\lim\limits_{m_k\rightarrow\infty}\|M_{\eta}^{m_k}(f)\|=0$.
Then
\begin{eqnarray*}
0\leftarrow
\|M_{\eta}^{m_k}(f)\|^{2}&=&\int_{\sigma(N)}|\eta^{m_k}f|^{2}d\mu \\
&=& \int_{\Delta_{1}\cap \Delta_{4}}|\eta^{m_k}f|^{2}d\mu +
\int_{\Delta_{2}\cap \Delta_{4}}|\eta^{m_k}f|^{2}d\mu \\
&\geq& \int_{\Delta_{1}\cap \Delta_{4}}|f|^{2}d\mu +
\int_{\Delta_{2}\cap \Delta_{4}}|\eta^{m_k}f|^{2}d\mu,
\end{eqnarray*}
and hence $\mu(\Delta_{1}\cap \Delta_{4})=0$. For any $m\in
\mathbb{N}$, there exist $k$ such that $m_k \leq m < m_{k+1}$.
Consequently,
\begin{eqnarray*}
||M_{\eta}^m(f)||^2 &=& \int_{\Delta_{2}\cap \Delta_{4}}|\eta^m
f|^{2}d\mu \\
&=& \int_{\Delta_{2}\cap \Delta_{4}}|\eta^{m_k}
f|^{2}|\eta^{m-m_{k}}|^2d\mu \\
&\leq& \int_{\Delta_{2}\cap \Delta_{4}}|\eta^{m_k} f|^{2}d\mu \\
&=& \|M_{\eta}^{m_k}(f)\|^{2}.
\end{eqnarray*}
Therefore, $\lim\limits_{m\rightarrow\infty}\|M_{\eta}^m(f)\|=0$.
\end{proof}

\begin{proposition}
Let $K$ be a compact operator on complex Hilbert space, then $K$ is
impossible to be Li-Yorke chaotic. Consequently, $K$ is neither
distributional chaotic nor norm-unimodal.
\end{proposition}

\begin{proof}
According to Riesz Decomposition Theorem, we have
$$K=\begin{matrix}\begin{bmatrix}
K_{1}\\
&K_{2}\\
\end{bmatrix}&
\begin{matrix}
H_1\\
 H_2\end{matrix}\end{matrix}$$
where $\sigma(K_{1})=\sigma(K)\cap \mathbb{D}$ and
$\sigma(K_{2})=\sigma(K)-\sigma(K_{1})$ .

Furthermore,
$$K=\begin{matrix}\begin{bmatrix}
K_{1}&*\\
&\widetilde{K_{2}}\\
\end{bmatrix}&
\begin{matrix}
H_1\\
  H_1^{\perp}\end{matrix}\end{matrix}$$
and $\sigma(\widetilde{K_2})=\sigma(K_{2})=\{\mu_1 , \mu_2 , \ldots
,\mu_l \}$.

Since $K$ is a compact operator , then

(1) there exists $0<\rho<1 , N\in \mathbb{N}$ such that for any
$x\in H_{1}$, $\|K^n(x)\|\leq \rho^n \|x\|$, for every $n\geq N$.

(2) $\widetilde{K_{2}}$ is similar to Jordan model
$J=\bigoplus_{i=1}^{l}\{\bigoplus_{j=1}^{k_i}J_{n_{j}^{i}}(\mu_i)\},$
where
$$
{J_{n}(\mu)=\begin{matrix}\begin{bmatrix}
\mu&1\\
&\ddots&\ddots\\
&&\ddots&1\\
&&&\mu\\
\end{bmatrix}&
\begin{matrix}
\end{matrix}\end{matrix}}_{(n\times n)}.
$$
Hence
$$
K \sim T=\begin{matrix}\begin{bmatrix}
K_1&*\\
&J\\
\end{bmatrix}&
\begin{matrix}
H_1\\
  H_1^{\perp}\end{matrix}\end{matrix}.
$$

Consequently, $T$ and $K$ are simultaneously Li-Yorke chaotic or
not. At present, it suffice to consider the condition of only one
Jordan block $J=J_{n}(\mu)$.

If $|\mu|>1$, one can use the technology of proposition \ref{spec}
to obtain the result.

Let $|\mu|=1$. Since the dimension of $H_{1}^{\perp}$ is finite,
then for each $y\in H_{1}^{\perp}$,
$$
y=y_1 e_1 +y_2 e_2 +
\ldots +y_n e_n,
$$
where $\{e_1 , e_2 ,\ldots ,e_n\}$ is an
orthonormal basis of $H_{1}^{\perp}$.

For each $z\in H$, there is a unique decomposition $z=x+y$ where
$x\in H_{1}$ and $y\in H_{1}^{\perp}$. Claim that $y=0$ if $0\in
\omega(z)$. Suppose $y\neq0$. There must be $i$ such that $y_i \neq
0 , \ y_{i+1}=y_{i+2}=\ldots=y_n =0$. Then
\begin{eqnarray*}
&&\|T^m(z)\|^2 \geq \|J^{m}(y)\|^2=\|\begin{matrix}\begin{bmatrix}
C_{m}^0 \mu^m&C_{m}^1 \mu^{m-1}&\cdots&C_{m}^{n-1} \mu^{m-n+1}\\
&C_{m}^0 \mu^m&\cdots&C_{m}^{n-2} \mu^{m-n+2}\\
&&\ddots&\vdots\\
&&&C_{m}^0 \mu^m\\
\end{bmatrix}&
\begin{matrix}
\end{matrix}\end{matrix} \begin{matrix}\begin{bmatrix}
y_1\\
y_2\\
\vdots\\
y_n
\end{bmatrix}&
\begin{matrix}
\end{matrix}\end{matrix}\|^2 \\
&&=|C_{m}^0 \mu^m y_1+C_{m}^1\mu^{m-1}y_2+\ldots+C_{m}^{n-1}
\mu^{m-n+1}y_n|^2+ \\
& &|C_{m}^0 \mu^m y_2 + C_{m}^1\mu^{m-1}y_3+\ldots+C_{m}^{n-2}
\mu^{m-n+2}y_n|^2 +\ldots+|C_{m}^0 \mu^m
y_n|^2 \\
&& \geq |C_{m}^0 \mu^m y_i +
C_{m}^1\mu^{m-1}y_{i+1}+\ldots+C_{m}^{n-i}
\mu^{m-n+i}y_n|^2\\
&& = |y_i|^2.
\end{eqnarray*}
It is a contradiction to $0\in \omega(z)$.

Consequently, if $0\in \omega(z)$, we have
$\lim\limits_{m\rightarrow\infty}\|T^m(z)\|=\lim\limits_{m\rightarrow\infty}\|K_{1}^m(x)\|=0$.
Therefore, $T$ is impossible to be Li-Yorke chaotic, and so is $K$.
Furthermore, $K$ is neither distributional chaotic nor
norm-unimodal.
\end{proof}

\begin{proposition}
Let K be a compact operator on complex Hilbert space, $\lambda \in
\mathbb{C}$. Then $\lambda I+K$ is not norm-unimodal.
\end{proposition}

\begin{proof}
If $|\lambda|\neq 1$, then the result is followed by proposition
\ref{spec}. In fact, $\lambda I+K$ is not Li-Yorke chaotic
currently.

Now let $|\lambda|=1$. Let $\{\mu_{n}\}_{n=1}^{M}$, $1\leq M\leq
\infty$, be the spectra of $\sigma(\lambda I+K)\cap
{\mathbb{D}^{-}}^{c}$. Since $K$ is a compact operator, then
according to Riesz Decomposition Theorem
$$\lambda I+K=\begin{matrix}\begin{bmatrix}
T_{1}\\
*&T_{2}\\
*&*&T_{3}\\
\vdots&\vdots&\vdots&\ddots \\
*&*&*&\ldots&T_{\infty}
\end{bmatrix}&
\begin{matrix}
H_1\\
 H_2\\
  H_3\\
  \vdots\\
   H_{\infty}\end{matrix}\end{matrix}$$

where $H_1 , H_2 , \ldots , H_n , \ldots$ are such subspaces that
$\bigoplus_{j=1}^{n}H_j$ coincides with the Riesz subspace of
$(\lambda I+K)^{*}$ corresponding to the clopen subset
$\{\overline{\mu_1} , \overline{\mu_2} , \ldots ,
\overline{\mu_n}\}$ of $\sigma((\lambda I+K)^*)$,
$\sigma(T_n)=\{\mu_n \}$  and $H_\infty=H\ominus \{\bigoplus_{1\leq
n\leq M}H_n \}$. It is not difficult to check that
$\sigma(T_\infty)\subseteq \sigma(\lambda I+K)-\{\mu_n\}_{n=1}^M$.

Suppose $\lambda I+K$ is norm-unimodal, i.e., there is a constant $r
>1$ such that for any $m\in\mathbb{N}$, there exists $x\in
H$ satisfying
$$ \lim\limits_{k\rightarrow\infty}\|(\lambda I+K)^kx\|=0, \ \ \  and
 \ \ \ \| (\lambda I+K)^ix\|\geq r^i\|x\|, \ \
 i=1,2,\ldots,m.$$
For such $x\in H$, write $x=(\bigoplus_{n=1}^{M}x_n)\oplus
x_\infty$. Then $\lim\limits_{k\rightarrow\infty}\|(\lambda
I+K)^kx\|=0$ implies $\bigoplus_{n=1}^{M}x_n=0$; and $||T_\infty^i
x||=||(\lambda I+K)^ix||\geq r^i ||x|| , 1\leq i \leq m$, implies
$r(T_\infty)\geq r$. However, this is impossible since
$\sigma(T_\infty)\subseteq \mathbb{D}^{-}$.
\end{proof}

In the research of hypercyclicity, Herrero and Wang \cite{Her} gave
a surprising result.

\begin{proposition}[\cite{Her}]
For any $\epsilon>0$, there is a small compact operator
$\|K_{\epsilon}\|<\epsilon$ such that $I+K_{\epsilon}$ is
hypercyclic.
\end{proposition}
Correspondingly, we obtain a similar result for distributional
chaos. Although $I+K$ can't be norm-unimodal, it may hold Weakly
Distributionally Chaotic Criterion.

\begin{proposition}
For any $\epsilon>0$, there is a small compact operator
$\|K_{\epsilon}\|<\epsilon$ such that $I+K_{\epsilon}$ is
distributionally chaotic.
\end{proposition}

\begin{proof}
Without losses, assume $\mathcal {H}$ be a separable complex Hilbert
space. Given any $\epsilon>0$. Let $C_i$ be a sequence of positive
numbers increasing to $+\infty$. For each $i\in \mathbb{N}$, set
$\epsilon_i=4^{-i}\epsilon$. Then we can select $L_i$ such that
$(1+\epsilon_i)^{L_i}\geq\sqrt{2}C_i$. Moreover, choose $m_i$ such
that $\frac{L_i}{m_i}<\frac{1}{i}$.

Write $n_i=2m_i$. We can obtain a orthogonal decomposition of
Hilbert space $\mathcal {H}=\bigoplus_{i=1}^{\infty}H_i$, where
$H_i$ is $n_i$-dimensional subspace. Define operators on each $H_i$
as follows,
$$
S_{i}={\begin{matrix}\begin{bmatrix}
0&2\epsilon_{i}\\
&\ddots&\ddots\\
&&\ddots&2\epsilon_{i}\\
&&&0\\
\end{bmatrix}&
\begin{matrix}
\end{matrix}\end{matrix}}_{(n_{i}\times n_{i})}
 \ , \ \ \ \
K_{i}={\begin{matrix}\begin{bmatrix}
-\epsilon_{i}&2\epsilon_{i}\\
&\ddots&\ddots\\
&&\ddots&2\epsilon_{i}\\
&&&-\epsilon_{i}\\
\end{bmatrix}&
\begin{matrix}
\end{matrix}\end{matrix}}_{(n_{i}\times n_{i})}.
$$
Then
$$
I_{i}+K_{i}={\begin{matrix}\begin{bmatrix}
{1-\epsilon_{i}}&2\epsilon_{i}\\
&\ddots&\ddots\\
&&\ddots&2\epsilon_{i}\\
&&&{1-\epsilon_{i}}\\
\end{bmatrix}&
\begin{matrix}
\end{matrix}\end{matrix}}_{(n_{i}\times n_{i})}
= \ \ (1-\epsilon_{i})I_{i}+S_{i}.
$$
Let $x_i=(1,1,\ldots,1)\in {H}_i$. We have for $1\leq n \leq m_i$,
\begin{eqnarray*}
&& \|(I_{i}+K_{i})^{n}(x_i)\|
\\
&=& \|((1-\epsilon_{i})I_{i}+S_{i})^{n}(x_i)\|
\\
&=& \|(\sum\limits_{k=0}^{n}C_{n}^{k}(1-\epsilon_{i})^k
{S_i}^{n-k}){x_i}\| \\
&\geq&
\|(\underbrace{\sum\limits_{k=0}^{n}C_{n}^{k}(1-\epsilon_{i})^k
{(2\epsilon_{i})}^{n-k}, \ldots ,
\sum\limits_{k=0}^{n}C_{n}^{k}(1-\epsilon_{i})^k
{(2\epsilon_{i})}^{n-k}}_{m_i}, 0, \ldots, 0)\| \\
&=& \sqrt{m_i}(1+\epsilon_{i})^n \\
&=& \frac{(1+\epsilon_{i})^n }{\sqrt{2}}\|x_i\|.
\end{eqnarray*}
Consequently,
\begin{eqnarray*}
&&\frac{\# \{0\leq k\leq m_i-1; \| (I_{i}+K_{i})^{k}x_i \|\geq
C_i\|x_i\| \} }{m_i} \\
&\geq& \frac{\# \{L_{i}, L_i+1, \ldots, m_i-1 \} }{m_i} \\
&=& 1-\frac{L_i}{m_i}.
\end{eqnarray*}
Notice $K_i$ is of finite rank and $\|K_i\|\leq 4^{1-i}\epsilon$.
Hence, $K_{\epsilon}=\bigoplus_{i=1}^{\infty}K_i$ is a compact
operator on $\mathcal {H}$ and $\|K_{\epsilon}\|<\epsilon$. Since
$I+K_{\epsilon}=\bigoplus_{i=1}^{\infty}(I_i+K_i)$, then for
previous $x_i$ seemed as in $\mathcal {H}$,

$(WNU1)$ \ \
$\lim\limits_{k\rightarrow\infty}\|{(I+K_{\epsilon})}^kx_i\|=0$
since $r(I+K_{\epsilon})<1$.

$(WNU2)$  \ \ The sequence of positive integers $m_i$ increasing to
$+\infty$ satisfies
\begin{eqnarray*}
&&\lim\limits_{i\rightarrow\infty}\frac{\# \{0\leq k\leq m_i-1; \|
(I+K_{\epsilon})^{k}x_i \|\geq C_i\|x_i\| \} }{m_i}
\\
&=&\lim\limits_{i\rightarrow\infty}\frac{\# \{0\leq k\leq m_i-1; \|
(I_{i}+K_{i})^{k}x_i \|\geq C_i\|x_i\| \} }{m_i} \\
&=&\lim\limits_{i\rightarrow\infty}1-\frac{L_i}{m_i}=1.
\end{eqnarray*}
Therefore, $I+K_{\epsilon}$ is distributionally chaotic  by theorem
\ref{W-D-C-C}.
\end{proof}
\begin{remark}
From the construction above, we can see that distributional chaos
isn't preserved under compact perturbations for bounded linear
operators. The previous operator $I+K_{\epsilon}$ is a
counterexample. In fact, $(I+K_{\epsilon})-\widetilde{K_i}$ is not
distributionally chaotic, where
$\widetilde{K_i}=(\bigoplus_{j=1}^{i}0_{H_j})\bigoplus(\bigoplus_{j=i+1}^{\infty}K_j)$
is a compact operator with norm less than $4^{-i}\epsilon$.
\end{remark}


\begin{thebibliography}{000}

\bibitem{Smital2} F.Balibrea, J. Sm\'{\i}tal and M. \v{S}tef\'{a}nkov\'{a},  The three versions of
distributional chaos, \textit{Chaos, Solitons \& Fractals, Volume
23, Issue 5,} 2005,  1581-1583.

\bibitem{Davidson}K. R. Davidson,  Nest Algebras, Longman Scientific \& Technical, Essex, 1988.

\bibitem{Duan} J. Duan, X. C. Fu, P. D. Liu and A. Manning, A linear chaotic quantum
harmonic oscillator, \textit{Appl. Math. Lett. 12(1)}, 1999, 15-19.


\bibitem{Fu} X. C. Fu, J. Duan, Infinite-dimensional linear dynamical systems
with chaoticity, \textit{J. Nonlinear Sci. 9(2)}, 1999, 197-211.

\bibitem{Gim} F. Mart\'{\i}nez-Gim\'{e}nez, P. Oprocha and A. Peris, Distributional chaos for backward
shifts, \textit{Journal of Mathematical Analysis and Applications,
Vol. 351(2)}, 2009, 607-615.

\bibitem{Grosse} K. G. Grosse-Erdmann,
 Recent developments in hypercyclicity , \textit{Rev. R. Acad. Cien. Serie A. Mat., RACSAM Vol. 97 (2)} 2003,
 273-286.

\bibitem{Her} D. Herrero and Z. Wang, Compact perturbations of hypercyclic and
supercyclic operators, \textit{Indiana Univ. Math. J. 39} 1990,
819-829.

\bibitem{Hbz} Bingzhe Hou, Puyu Cui and Yang Cao, Chaos for Cowen-Douglas operators, submitted, available at
\href{http://cn.arxiv.org/abs/0903.4246}{http://cn.arxiv.org/abs/0903.4246}.

\bibitem{L-Y} T. Y. Li and J. A. Yorke, Period three implies chaos, \textit{Amer.
Math. Monthly 82(10)}, 1975, 985-992.

\bibitem{Liao1}Gongfu Liao,
Lidong Wang and Xiaodong Duan, A chaotic function with a
distributively scrambled set of full Lebesgue measure,
\textit{Nonlinear Analysis: Theory, Methods \& Applications, Vol.
66(10)},2007, 2274-2280.

\bibitem{Liao2} Gongfu Liao, Zhenyan Chu and Qinjie Fan, Relations between mixing and distributional
chaos, \textit{Chaos, Solitons \& Fractals}, In Press, Corrected
Proof, Available online 16 September 2008.


\bibitem{Opr} P. Oprocha, A quantum harmonic oscillator and strong chaos, \textit{J.
Phys. A 39(47),} 2006, 14559-14565.

\bibitem{Rolewicz} S. Rolewicz, On orbits of elements, \textit{Studia Math., Vol. 32}
1969, 17-22.

\bibitem{S-S} B. Schweizer and J. Sm\'{\i}tal, Measures of chaos and a spectral
decomposition of dynamical systems on the interval, \textit{Trans.
Amer. Math. Soc. 344(2)}, 1994, 737-754.

\bibitem{Shapiro} J. H. Shapiro, Notes on Dynamics of Linear
Operators, unpublished Lecture Notes, available at
\href{http://www.math.msu.edu/\~{}shapiro}{http://www.math.msu.edu/\~{}shapiro}.


\bibitem{Smital1} J. Sm\'{\i}tal, M. \v{S}tef\'{a}nkov\'{a}, Distributional chaos for triangular
maps, \textit{Chaos, Solitons \& Fractals, Volume 21, Issue 5},
2004, 1125-1128.



\bibitem{Wang} Hui Wang, Gongfu Liao and Qinjie Fan, Substitution systems and the
three versions of distributional chaos, \textit{Topology and its
Applications, Volume 156, Issue 2}, 2008, 262-267.

\end{thebibliography}
\end{document}